\input xy
\xyoption{all}
\documentclass[12pt]{amsart}
\usepackage{amscd,amssymb,graphics}
\usepackage{relsize}
\usepackage{amsfonts}
\usepackage{amsmath}
\usepackage{enumerate}

\newtheorem{thm}{Theorem}[section]
\newtheorem{fact}[thm]{Fact}
\newtheorem{corol}[thm]{Corollary}
\newtheorem{lemma}[thm]{Lemma}
\newtheorem{prop}[thm]{Proposition}

\newtheorem{quest}[thm]{Question}
\newtheorem{pr}[thm]{Problem}

\theoremstyle{definition}
\newtheorem{defin}[thm]{Definition}

\theoremstyle{remark}
\newtheorem{remark}[thm]{Remark}

\newtheorem{ex}[thm]{Example}
\newtheorem{example}[thm]{Example}

\newcommand{\ben}{\begin{enumerate}}
\newcommand{\een}{\end{enumerate}}
\newcommand{\bit}{\begin{itemize}}
\newcommand{\eit}{\end{itemize}}

\def\R {{\mathbb R}}

\def\nbd{{neighborhood}}

\def\N{{\mathbb N}}
\def\T{{\mathbb T}}

\def\RUC{{\hbox{RUC\,}^b}}

\def\Z {{\mathbb Z}}
\def\U{{\mathcal U}}

\def\Homeo{{\mathrm{Homeo}}\,}

\def\eps{{\varepsilon}}
\def\K {\mathcal K}

\def\QED{\nobreak\quad\ifmmode\roman{Q.E.D.}\else{\rm Q.E.D.}\fi}

\def\GL{\operatorname{GL}}

\def\a {\alpha}

\newcommand{\sk}{\vskip 0.3cm}

\def\Asp{\operatorname{Asp}}
\def\RUC{\operatorname{RUC}}

\def\LUC{\operatorname{LUC}}
\def\UC{\operatorname{UC}}
\def\WAP{\operatorname{WAP}}

\def\Tame{\operatorname{Tame}}
\def\Iso{\operatorname{Iso}}

\def\SL{\operatorname{SL}}

\usepackage[hidelinks]{hyperref}

\begin{document}

\title[]{Topological group actions by group automorphisms and Banach representations} 

\dedicatory{Dedicated to my friend Vladimir Pestov on the occasion of
	his 65th birthday}

\author[]{Michael Megrelishvili}
\address{Department of Mathematics,
	Bar-Ilan University, 52900 Ramat-Gan, Israel}
\email{megereli@math.biu.ac.il}
\urladdr{http://www.math.biu.ac.il/$^\sim$megereli}

\date{2023, February 14} 
\subjclass[2020]{Primary 37Bxx, 22Dxx; Secondary 54H15, 46Bxx}

\keywords{Banach representation, conjugation action, equivariant compactification, enveloping semigroup, hereditarily non-sensitive, Rosenthal Banach space, tame dynamical system, weakly mixing}


\thanks{This research was supported by a grant of the Israel Science Foundation (ISF 1194/19)} 
 
\begin{abstract}  
	   We study Banach representability for actions of topological groups on groups by automorphisms 
	   (in particular, an action of a group on itself by conjugations). Every such action is Banach representable on some Banach space. The natural question is to examine when we can find representations on low complexity Banach spaces.  
	
	In contrast to the standard left action of a locally compact second countable group $G$ on itself, the conjugation action need not be reflexively representable even for $\SL_2(\mathbb{R})$.  
	The conjugation action of $\SL_n(\mathbb{R})$ is not Asplund representable for every $n \geq 4$. 
	The linear action of $\GL_n(\mathbb{R})$ on $\R^n$, for every $n \geq 2$, is not representable on Asplund Banach spaces. On the other hand, this action is representable on a \textit{Rosenthal Banach space} (not containing an isomorphic copy of $l_1$).  
 
	 The conjugation action of a locally compact group need not be Rosenthal representable (even for Lie groups).  
	 As a byproduct we obtain some counterexamples about Banach representations of homogeneous $G$-actions $G/H$.  
\end{abstract}

\maketitle
 
\setcounter{tocdepth}{1}
 \tableofcontents
  
\section{Introduction}

\subsection{Main question} 

 To every Banach space $V$ one may associate the topological group 
 $\Iso(V)$ of linear
isometries (in its strong operator topology) and its canonical dual action on the weak-star compact unit ball $B_{V^*}$ of the dual Banach space $V^*$.   
The main idea of the present work is to study which 
actions $G \times X \to X$ \textit{by automorphisms on topological groups} $X$ can be obtained as a ``subaction" 
of $\Iso(V) \times B_{V^*} \to B_{V^*}$ 

\begin{quest} \label{problem1} 
Which actions of a topological group $G$ on a topological group $X$ by group automorphisms can be represented on low complexity Banach spaces $V$ (e.g., Hilbert, reflexive, Asplund, Rosenthal)?			
In particular, what about the actions of groups on itself by conjugations?
\end{quest}

 This is a particular case of a question about general actions which  
 leads to an interesting hierarchy of topological groups and their actions. For some concrete results in this direction we refer to  \cite{Me-nz,Me-opit,GM-HNS,GM-rose,GM-survey14,GM-MoreTame18,Ibar,GM-UltraHom19}.  

\subsection{Preliminaries} 

All topological spaces are assumed to be Tychonoff. $X$ is a 
$G$-\textit{space} will mean that we have a continuous action $G \times X \to X$ of a topological group $G$ on a topological space $X$.
	
 The dual $V^*$ of a Banach space $V$ is the Banach space of all continuous linear functionals on $V$. 
 
 \begin{remark} \label{r:daction} The topological group 
 	 $G:=\Iso(V)$ of linear 
 	isometries 
  acts naturally on the dual space $V^*$ as follows: 
  $$
  \Iso(V) \times V^* \to V^*, \ \ ((gu^*)(v)=u^*(g^{-1}v)) \ \ \ g \in G, v \in V, u^* \in V^*. 
  $$
  Moreover, the induced action on every norm bounded $G$-invariant subset $Y$ of $V^*$ (e.g., on the unit ball $B_{V^*}$ of $V^*$) is continuous where $Y$ is endowed with the weak-star topology. 
 \end{remark}

 \begin{defin} \label{d:repr} \cite{Me-nz,GM-survey14} A 
 \textit{representation} of a $G$-space $X$ on a Banach space $V$ is an equivariant continuous pair $(h,\a)$, where $h\colon G \to \Iso(V)$ is a continuous homomorphism  and $\a\colon X \to V^*$ is a weak-star continuous bounded $G$-equivariant map.   
 \textit{Proper representation} means that $\a$ is a topological embedding. A $G$-space $X$ is \textit{representable} in a class $\K$ of Banach spaces if there exists a proper $G$-representation of $X$ on some $V \in \K$.  
 \end{defin}

Reflexive representability of the $G$-space $X$ means that in the setting of Definition \ref{d:repr} one may choose a  reflexive space $V$.  
Hilbert, Asplund and Rosenthal representability can be similarly defined. 

Recall that a Banach space $V$ is {\em Asplund\/} if the dual $W^*$ of every separable Banach subspace $W$ of $V$ is separable. 
The definition of Rosenthal Banach spaces comes directly from 
Rosenthal's celebrated dichotomy, \cite{Ros74}. According to this dichotomy every bounded sequence in a Banach space either has a weak Cauchy subsequence or admits
a subsequence equivalent to the unit vector basis of $l_1$ (an
$l_1$-\emph{sequence}). 
A Banach space $V$ does not contain an  $l_1$-sequence
(equivalently, does not contain an isomorphic copy of $l_1$) if
and only if every bounded sequence in $V$ has a weak-Cauchy
subsequence \cite{Ros74}. 
As in \cite{GM-rose,GM-survey14,GM-TC}, we call a Banach space satisfying these
equivalent conditions a {\em Rosenthal space}.  
Every reflexive space is Asplund and every Asplund space is Rosenthal.

A $G$-\textit{compactification} of a $G$-space $X$ is a $G$-equivariant continuous dense map $\nu\colon X \to Y$ into a compact 
$G$-space $Y$. If $\nu$ is a topological embedding then $\nu$ is said to be a \textit{proper} compactification. A $G$-space is $G$-\textit{compactifiable} if it admits a proper $G$-compactification. 

For every compact $G$-space $K$, there exists a canonical Banach $G$-representation on the Banach space $C(K,\R)=C(K)$ (see for example \cite{Pe-w}).  
Together with Remark \ref{r:daction} this explains the following basic fact. 

\begin{fact} \label{f:C=R}  
	A continuous action $G \times X \to X$ is Banach representable if and only if the $G$-space $X$ is $G$-compactifiable. 
\end{fact}

 Alexandrov \textit{1-point compactification} $X \cup \{\infty\}$ for a locally compact noncompact space $X$ is the smallest proper compactification of $X$. For every continuous topological group action 
 on $X$, the canonical extension on $X \cup \{\infty\}$ is continuous, \cite{Vr-can}. 
For more information about $G$-compactifications we refer to \cite{Vr-can,Vr-loccom78,Me-sing,Me-F,Me-opit,Pe-nbook,IbMe}. 

 \begin{remark} \label{r:repr-of-funct} 
 	For every representation $(h,\a)$ of a $G$-space $X$ on $V$, every vector $v \in V$ induces 
 	the function $f_v\colon X \to \R, x \mapsto \langle v,x \rangle$. This is an important source to obtain functions on $G$-spaces. For example, varying $v \in V$ and $V$ in the class of all reflexive Banach spaces we get the class of all \textit{reflexively representable functions} on $X$. By \cite{Me-nz}, it is exactly the algebra $\WAP(X)$ of all \textit{weakly almost periodic functions}. The algebras $\Asp(X), \Tame(X)$ can be similarly characterized, 
 	 using the corresponding classes of Banach spaces, i.e., Asplund and Rosenthal, \cite{GM-survey14, GM-rose}. 
 \end{remark}

By Fact \ref{f:C=R}, a necessary (and sufficient) condition of Banach representability of a $G$-space $X$ is the $G$-compactifiability. 
The corresponding algebra of all ``Banach representable functions" is $\RUC(X)$. Recall that a bounded continuous function $f\colon X \to \R$ is (generalized) \textit{right uniformly continuous}  (in short, RUC) if for every $\eps >0$ there exists a neighborhood $O$ of the unity $e \in G$ such that $|f(gx)-f(x)| < \eps$ for every $(g,x) \in O \times X$.

For every topological group $G$ we have canonically defined left action of $G$ on itself. 
In this case, $\RUC(G)$ is the usual algebra of all bounded right uniformly continuous functions on $G$. 
It is well known that $\RUC(G)$ separates the points and closed subsets and this action is always $G$-compactifiable. The \textit{greatest ambit} $\beta_G G$ of $G$ 
is its (proper) maximal $G$-compactification. 

\sk 
\textbf{Notation.} 
Consider now the action of $G$ on itself by conjugations 
$$
\pi_c \colon G \times G \to G, (g,x) \mapsto gxg^{-1}. 
$$
We denote by $G_c$ the $G$-space $G$ with the \textit{conjugation action} $\pi_c$. 
\sk

The importance of the action $\pi_c$ is understood. However its Banach $\K$-representability (in the sense of Definition \ref{d:repr}) has not been studied yet even for locally compact (classical) groups. 

\sk 
Important algebras 
of functions 
 on $G$, like $\WAP(G), \Asp(G), \Tame(G)$ 
are defined for 
the left regular action. 
Members of these algebras  are exactly the functions on $G$ which come as generalized matrix coefficients induced by representations 
$h \colon G \to \Iso(V)$ 
of $G$, where the Banach space $V$ is Rosenthal (respectively, Asplund and reflexive). See \cite{Me-nz, GM-HNS, GM-rose,GM-survey14}. 

For every locally compact topological group $G$, the algebra $\WAP(G)$ of all \textit{weakly almost periodic functions} separate points and closed subsets and $G$ 
admits a proper representation $h \colon G \hookrightarrow \Iso(V)$ 
 on a reflexive Banach space $V$. 
In fact, one may find a Hilbert representation of $G$ as it follows by a classical result of Gelfand and Raikov \cite[p. 314]{Gelf-Raikov}.   
If $G$, in addition, is separable metrizable, then we can require more; namely, that there exists even an equivariant representation, in the sense of Definition \ref{d:repr}, of the standard $G$-space (with left translations) $X:=G$ on a Hilbert space.  
This can be done using \cite[Lemma 4.5]{Me-hilb}. 

\sk 
\subsection{Main results and some open questions} 

It is a less known fact (see \cite{Me-sing,Me-F}) that the $G$-space $G_c$ 
(conjugation action) 
is Banach representable (in the sense of Definition \ref{d:repr})  for every topological group $G$. By Fact \ref{f:C=R}, it is equivalent to say that $G_c$ is $G$-compactifiable. More precisely, the \textit{Roelcke compactification} is a (proper) $G$-compactification of $G_c$ (Corollary \ref{t:Roelcke}). The significance of Roelcke compactification is now well understood due to several papers of V. Uspenskij \cite{UspComp,Uscurves} and many other authors. 

In contrast to the standard left action of a locally compact second countable group $G$ on itself, the conjugation action need not be reflexively representable even for natural matrix groups. This happens among others for the special linear group $\SL_2(\R)$; see Theorem \ref{t:SL(2)} which uses Grothendieck's \textit{double limit criterion}.  
Moreover, by Theorem \ref{t:SL} the $G$-space $G_c$ for $G:=\SL_n(\R)$ is not Asplund representable for every $n \geq 4$. 
Here we use the \textit{weak mixing argument} by S.G. Dani and S. Raghavan
\cite{DR} together with the concept of hereditarily non-sensitive actions, \cite{GM-HNS}. 
We should note that dynamical non-sensitivity is a necessary condition for Asplund representability of compact $G$-spaces by joint results with E. Glasner \cite{GM-HNS}. Moreover, hereditary non-sensitivity is a sufficient condition of Asplund representability of metrizable compact $G$-spaces. 
%

%

The natural linear action of $\GL_n(\R)$ on $\R^n$, for every $n \geq 2$, is not Asplund representable (Proposition \ref{t:GL} and Theorem \ref{t:SL}). On the other hand, this action is Rosenthal representable, Theorem \ref{t:LinAreTame}. Here we use the enveloping semigroup characterization of \textit{tame} compact dynamical systems (see Section \ref{s:tame}). Recall that the \textit{enveloping (Ellis) semigroup} of an action $G \times X \to X$ on a compact space $X$ is the pointwise closure of all $g$-translations $X \to X$ ($g \in G$) in the compact space $X^X$.

A compact $G$-system $X$ is said to be \textit{tame} (\textit{regular}, in terms of A. K\"{o}hler \cite{Ko}) if for every $f \in C(X)$ the orbit $fG$ 
does not contain a combinatorially independent sequence in the sense of Rosenthal \cite{Ros74}. Tame dynamical systems naturally occur in geometry, analysis and symbolic dynamics \cite{GM-MoreTame18}. They 
play an important role in view of a dynamical  
Bourgain-Fremlin-Talagrand dichotomy \cite{GM-survey14}, NIP-formulas in logic \cite{Ibar} and Todor\u{c}evi\'{c}' trichotomy in topology \cite{GM-TC}. 
In \cite{GM-UltraHom19} a generalized amenability was examined substituting, in the definition, the existence of a fixed point by some tame dynamical $G$-subsystem. 

Every hyperbolic toral automorphism defines a cascade which is not tame and hence not Rosenthal representable (Theorem \ref{t:Lebed}). Moreover, using V. Lebedev's recent result \cite{Leb}, one may show that the same is true for every infinite order automorphism of the torus $\T^n$. 
One of the conclusions is that 
 the conjugation action need not be Rosenthal representable even for Lie groups (Corollary \ref{c:notTame1}). This is unclear for $\SL_2(\R)$. 

Like
for 
 locally compact groups, also for non-archimedean second countable groups $G$,  the standard left action on itself is Hilbert representable. 
Corollary \ref{c:notTame} shows that there exist 
Polish non-archimedean 
locally compact topological groups  $G$ (which are \textit{elementary} in the sense of Wesolek \cite{Wesolek}) such that 
the conjugation action is not Rosenthal representable.

\begin{pr} 
	For the conjugation action $G \times G_c \to G_c$, study:
	\begin{enumerate}
		\item the greatest $G$-compactification $\beta_G G_c$ of $G_c$;   
		\item  the algebras $\RUC(G_c),\Tame(G_c), \Asp(G_c), \WAP(G_c)$; 
		\item when the conjugation action is Rosenthal representable. 
	\end{enumerate}

What about the following concrete groups: a) $\GL_2(\R)$, $\SL_2(\R)$; b) unitary group $\Iso(l_2)$; 
c) symmetric group $S_{\infty}$; d) $\Iso(\mathbb{U})$; e) $\Iso(\mathbb{U}_1)$? 
\end{pr}

Here $\mathbb{U}$ is the \textit{Urysohn universal metric space} (see, for example \cite{Pe-w}), $\mathbb{U}_1$ is the Urysohn sphere and $\Iso(M)$ means the topological group of all isometries with the pointwise topology. 

 As a byproduct we illuminate some counterexamples about Banach representations of homogeneous $G$-actions $G/H$ in Section \ref{s:homog}. Among others we prove (Corollary \ref{c:homogLie}) that 
there exists a two dimensional Lie group $G$ and its cocompact discrete subgroup $H:=\Z$, 
such that the compact two dimensional homogeneous $G$-space $G/H$ is not Rosenthal representable. 

\sk 
We hope that Banach $\K$-representability for the conjugation actions  
will foster some new ideas 
and 
open up interesting research lines in the realm of (Polish) topological groups even for the subclass of (classical)  locally compact second countable groups.

 \sk  
 \section{Some properties of actions by automorphisms}
 
\subsection*{Compactifiability}  

Let $\alpha \colon G \times X \to X$ be a continuous action  of a topological group $G$ on $X$. 
A topologically compatible uniform structure $\U$ on $X$ is said to be \textit{bounded} (see \cite{Br,Vr-can,Vr-loccom78,Pe-nbook}) if for every entourage $\eps \in \U$ there exists a neighbourhood $U \in N_e(G)$ such that $(x,ux) \in \eps$ for every $u \in U$ and every $x \in X$. If, in addition, every $g$-translation is a uniform map, then $\U$ is an  
\textit{equiuniformity} (in the terminology of \cite{Me-EqComp84}). According to Brook \cite{Br}, the Samuel compactification $(X,\U) \to sX$ of an equiuniformity $\U$ is a (proper) $G$-compactification of $X$. 

We say that $\U$ is \textit{quasibounded} (introduced in \cite{Me-EqComp84,Me-sing}), if for every $\eps \in \U$ there exist $\delta \in \eps$ and $U \in N_e(G)$ such that $(ux,uy) \in \eps$ for every $(x,y) \in \delta$ and $u \in U$. 

The $G$-compactifiability of $X$ is equivalent to the existence of a quasibounded uniformity on 
$X$. 
 As it was proved in \cite[p. 222]{Me-sing} and \cite{Me-b}, there exists a natural construction to obtain a bounded compatible uniformity  on $X$ for a given quasibounded uniformity on a $G$-space $X$.

\begin{fact} \label{t:aut} \cite{Me-F}  
	Let $G$ and $X$ both be topological groups, $\pi\colon G \times X \to X$ a continuous action by group automorphisms of $X$. Then the $G$-space $X$ is $G$-compactifiable (and it admits a Banach representation in the sense of Definition \ref{d:repr}). 
\end{fact}
\begin{proof} Since the given continuous action is by automorphisms of $X$, it is straightforward to observe that the right (and also left) uniformity on $X$ is quasibounded (but not always bounded). 
\end{proof}

	An additional possibility proving Fact \ref{t:aut} is to use Fact  \ref{f:SemidCoset} below taking into account that (according to J. de Vries \cite{Vr-can}) every coset $G$-space $G/H$ is $G$-compactifiable. The reason is that the right uniformity of $G/H$ is bounded. 

\sk
Recall that the \textit{Roelcke compactification} $\rho\colon G \to \rho(G)$ of a topological group $G$ is the compactification induced by the algebra 
$\UC(G)=\RUC(G) \cap \LUC(G)$, where $\LUC(G)$ is the algebra of all bounded left uniformly continuous functions. 

\begin{fact} \label{p:diag} \cite{Me-F} 
	Let $G$ be a (Hausdorff ) topological group. Consider the following continuous action of $P:=G \times G$ on $G$ 
	$$
	P \times G \to G,  (s,t)(g)=sgt^{-1}.
	$$
	Then the Roelcke compactification 
	is a proper $P$-compactification of $G$. 
\end{fact}
\begin{proof} One may easily verify that the 
	Roelcke uniformity $\UC(G)$ 
	on $G$ is always an equiuniformity on the $P$-space $G$ and hence,  according to \cite{Br}, its Samuel compactification $\rho \colon G \to \rho (G)$ is a (proper) $P$-compactification of $G$. 
\end{proof}

The action of $G$ on itself by conjugations is a subaction of the diagonal subgroup $G \simeq \Delta:=\{(g,g): g \in G\} \leq G \times G$ on $G$.  
In particular, we get 

\begin{corol} \label{t:Roelcke}  
	The Roelcke compactification  
	of $G$ 
	defines a proper $G$-compactification of the $G$-space $G_c$. So, 
	\begin{enumerate}
		\item $G_c$ is a $G$-compactifiable $G$-space; 
		\item the Roelcke compactification $\rho: G_c \to \rho(G)$ is a $G$-factor of $\beta_G (G_c)$; 
		\item $\UC(G) \subset \RUC(G_c)$. 
	\end{enumerate} 
\end{corol}

\begin{remark} \label{r:semid} 
Let $\a \colon G \times X \to X$ be a continuous action of a topological group $G$ on a topological group $X$ by group automorphisms. Denote by $X \rtimes_{\a} G$ the corresponding topological semidirect product. As usual, identify $G$ with $\{e_X\} \times G$ and $X$ with $X \times \{e_G\}$. Then $X$ is a normal subgroup of $X \rtimes_{\a} G$ and the action $\a$  of $G$ on $X$ is a subaction of the conjugation action of $X \rtimes_{\a} G$ on itself. 
\end{remark}

One of the conclusions of Remark \ref{r:semid} is that Corollary \ref{t:Roelcke}.1 infers back Fact \ref{t:aut}. 

\begin{fact} \label{f:SemidCoset} \cite[Lemma 1.1]{Me-F}
	Let $\alpha \colon G \times X \to X$ be a continuous action by group automorphisms and $P:=X \rtimes_{\a} G$  be the corresponding topological semidirect product. 
	Then the triple $(G,X,\alpha)$ naturally is embedded into the homogeneous action $(P, P/G, \alpha_*)$, where $\alpha_*$ is the natural action of $P$ on $P /G$. 
\end{fact}
\begin{proof}
	The mapping $j \colon X \to P/G, \ j(x)=xG$ is a restriction of the natural projection $P \to P/G$ on $X \subset P$. According to \cite[Prop. 6.17(a)]{RD}, 
	$j$ is a homeomorphism. Moreover, it is straightforward to check that 
	$$(e,g)j(x)=(e_X,g)((x,e_G)G)=(gx,e_G)G=j(gx).$$ So, the restriction of $\alpha_*$ on $G \times X:=G \times j(X)$ is $\alpha$. 
\end{proof} 

Roughly speaking, every action by automorphisms is a subaction of a coset $G$-space. The converse is not true in general. By Example \ref{e:NonAutComp} below the homogeneous action of the full homeomorphism Polish group $\Homeo (S)$ on the sphere $S$ is not a part of any continuous action by group automorphisms. 

\sk 
\subsection*{Which actions are automorphizable?}

For every $G$-space $X$ there exists the \textit{free topological $G$-space} $F_G(X)$ of $X$. 
This concept was  introduced in \cite{Me-F} and has several nice applications;  among others for epimorphism problems.  
Resolving a longstanding important problem by K.
Hofmann, Uspenskij \cite{Usp-epic} has shown that in the category of 
Hausdorff topological groups epimorphisms need not have a dense range. 
Dikranjan and Tholen \cite{DiTh} gave a rather direct proof.
Pestov gave 
a beautiful useful criterion \cite{Pest-epic,Pe-w} 
which is based on the concept of a free topological $G$-group. 
More precisely, the inclusion $i\colon  H
\hookrightarrow G$ of topological groups is an epimorphism if and only if the
free topological $G$-group $F_G(X)$ of the coset $G$-space
$X:=G/H$ is trivial. Triviality means isomorphic to the cyclic discrete group (``as trivial as possible").

In contrast to the case of trivial $G$ (when $F_G(X)$ is just the usual free topological group $F(X)$), the universal $G$-morphism $i \colon X \to F_G(X)$ need not be a topological embedding even for compact $G$-space $X$ as it follows directly from the following 

\begin{ex} \label{e:NonAutComp}  \cite{Me-F}
	Not every compact $G$-space $K$ is a subaction of an action by group automorphisms.  
	For example, the cube $K=[0,1]^n$ or the $n$-dimensional sphere $K=S_n$ ($n \in \N$) with the homeomorphisms group $G=\Homeo(K)$, which is Polish in the compact-open topology.  
\end{ex}

In the case of the circle $K=S$, the corresponding homogeneous action of the Polish group $G=\Homeo(S)$ on $S$ can be identified with the compact coset $G$-space $G/H$, where $H=St(z)$ is the stabilizer for a point $z \in S$. The corresponding free topological $G$-space $F_G(G/H)$ is trivial (by \cite{Me-F}) and Pestov's criterion implies that the closed inclusion $H \hookrightarrow G$ is an epimorphism.

 Moreover, a much stronger result follows by the earlier paper of Uspenskij \cite{Usp-epic}. Namely, in fact, for every compact connected manifold $X$, its Polish homeomorphism group $G=\Homeo(X)$ and a stability subgroup $H=St(z)$, the embedding  $H \hookrightarrow G$ is an epimorphism. Equivalently, any continuous $G$-map from $X$ into any $G$-group is constant. A self-contained elegant explanation of the latter fact can be found in a recent work by Pestov and Uspenskij \cite[page 5]{PU}.

\begin{remark} \label{r:SufficCond} \ 
	\begin{enumerate}
		\item It is well known that such examples are impossible for locally compact $G$ because, in this case, every Tychonoff $G$-space $X$ is even $G$-linearizable on a locally convex linear $G$-space.
		\item  Another sufficient condition is (uniform) $\U$-equicontinuity of the action with respect to some compatible uniformity on $X$. One may assume that $\U$ is generated by a system of $G$-invariant pseudometrics $\{\rho_i: i \in I\}$ with $\rho_i \leq 1$. Now observe that the  Arens-Eells embedding defines a $G$-linearization of $X$ into a locally convex $G$-space. 
	\end{enumerate}

For details see \cite{Me-F} and \cite{Me-b}.  
\end{remark}

\sk 
\section{Representations on low complexity Banach spaces}

\subsection{Reflexive representability} 

 According to a classical definition, a continuous bounded function $f \in C_b(X)$  on a $G$-space $X$ is said to be \textit{weakly almost periodic} (WAP)  if the weak closure of the orbit 
 $fG=\{fg:g \in G\}$  is weakly compact in the Banach space $C_b(X)$ (with the sup-norm).  A compact $G$-space $X$ is WAP if every $f \in C(X)$ is WAP.

\begin{fact} \label{f:Ref-WAP} \cite{Me-nz} 
 Every reflexively representable $G$-space is embedded into a WAP compact $G$-space. 
		A compact metric $G$-space $X$ is reflexively representable if and only if $X$ is a WAP $G$-system. 
\end{fact}

The following result, based on Grothendieck's criterion, can be derived by combining \cite[Theorem 4.6]{Me-nz} and \cite[Fact 2.4]{Me-nz}

\begin{fact} \label{f:WAP} \cite{Me-nz}
	Let $X$ be a (not necessarily compact) $G$-space and $f \in \RUC(X)$. The following conditions are equivalent: 
	\begin{enumerate}
		\item $f \in \WAP(X)$. 
		\item $f$ has Grothendieck's double limit property. 
	\item $f$ is reflexively representable (Remark \ref{r:repr-of-funct}). 
	\item $f$ comes from a $G$-compactification $\nu \colon X \to Y$ of $X$ such that $Y$ is WAP. 
	\end{enumerate} 
\end{fact}

\sk 
\begin{thm} \label{t:SL(2)} 
Let $G:=\SL_n(\R)$, where $n \geq 2$. 
Then the conjugation $G$-space $G_c$ is not reflexively representable. 
\end{thm}
\begin{proof} Assume, to the contrary, that the $G$-space $G_c$ is reflexively representable. Then by Fact \ref{f:Ref-WAP} there exists a proper $G$-compactification $G_c \hookrightarrow X$ such that the $G$-space $X$ is WAP. Then every $G$-factor of $X$ is again WAP. In particular, Alexandrov's 1-point compactification (smallest \textit{proper} $G$-compactification of $G_c$) $Y:=G_c \cup \{\infty\}$ is WAP. 
	
	We claim that $Y$ is not WAP. 
	It is enough to show that for every  compact \nbd \ $U$ of the identity $e \in G$ and for every continuous bounded function $f \colon G \to \R$ with $f(e)=1$ and $f(x)=0$ for every $x \notin U$, we have $f \notin \WAP(G_c)$. 
	By the Grothendieck double limit criterion for $G$-spaces (as in Fact \ref{f:WAP}), it suffices to show that there exist two sequences $g_n \in G$ and $x_m \in G_c$ such that the double sequence $f(g_nx_mg_n^{-1})$ ($n,m \in \N$) has distinct double limits. 
	
	Now for $n=2$, define the following sequences (for general $n \geq 2$ the proof is similar):    
	$$g_n:= \left(\begin{array}{cc}
	n^{-1} & 0  \\
	0 & n 
	\end{array}\right), \ \ x_m:= \left( \begin{array}{cc}
	1 & m  \\
	0 & 1  
	\end{array}\right).$$
		Then 
	$$g_nx_mg_n^{-1}=\left(\begin{array}{cc}
	1 & \frac{m}{n^2}  \\
	0 & 1  
	\end{array}\right)$$ 
	$$ \lim_n \lim_m (g_n x_m g_n^{-1})=\infty \ \ \ \neq \ \ \
	\lim_m \lim_n (g_nx_m g_n^{-1})= 
	\left( \begin{array}{cc}
	1 & 0  \\
	0 & 1  
	\end{array}\right).$$
	Hence, 
	$$ \lim_n \lim_m f(g_n x_m g_n^{-1})=0 \ \ \ \neq \ \ \ 1=
	\lim_m \lim_n f(g_nx_m g_n^{-1}).$$ 
	\end{proof}

\begin{prop} \label{t:n=1} 
	Let $\R^{\times}$ be the multiplicative group of all 
	nonzero reals. Then 
	\begin{enumerate}
		\item 	the natural action $\a \colon \R^{\times} \times \R \to \R$ is not reflexively representable; 
		\item for the "$ax+b$-group" $G:=\R \rtimes_{\a} \R^{\times}$ the $G$-space $G_c$ is not reflexively representable;  
		\item let $X=\R \cup \{\infty\}$ be the 1-point compactification of $\R$. Then the enveloping semigroup $E(X)$ of the action of $\R^{\times}$ on $X$ is the semigroup $\R^{\times} \cup \{0\} \cup \{\infty\}$, with $0 \cdot \infty=\infty$ and $\infty \cdot 0=0$ (other cases are understood). 
	\end{enumerate}
\end{prop}
\begin{proof}
	(1) Choose $g_n:=n^{-1}, x_m:=m$, where $n,m \in \N$. Then 
	$$ \lim_n \lim_m (g_n x_m)=\lim_n \lim_m \frac{m}{n}=\infty \ \ \ \neq \ \ \
	0= \lim_m \lim_n (g_n x_m)=\lim_m \lim_n \frac{m}{n}.$$
	The rest is similar to the proof of Theorem \ref{t:SL(2)}. 
	
	(2) This follows from (1) because $\a$ is a subaction of the action $G \times G_c \to G_c$ by Fact \ref{f:SemidCoset}. 
	
	(3) Straightforward. 
	\end{proof}

\sk  
\subsection{Asplund representability}

First recall the classical concept of non-sensitivity. 
An action of $G$ on a uniform space $(X,\U)$ is said to be \textit{non-sensitive} 
if for every entourage $\eps \in \U$ there exists a nonempty open subset $O$ in $X$ such that $gO$ remains $\eps$-small for every $g \in G$. 

According to \cite{GM-HNS}, 
\textit{hereditarily non-sensitive} (in short, HNS) means that every (equivalently, every closed) $G$-subspace $Y$ of $X$ is non-sensitive with respect to the induced subspace uniformity.

\begin{fact} \label{l:fromHNS} \
	\begin{enumerate}
		\item \cite{GM-HNS} Every Asplund representable $G$-space is embedded into a HNS compact $G$-space. 
	 A compact metric $G$-space $X$ is 
		Asplund representable if and only if $X$ is HNS. 
		\item 
		\cite{GMU}   
		A compact metric $G$-space $X$ is HNS if and only if the enveloping semigroup $E(X)$ is metrizable. 
	\end{enumerate}
\end{fact}

Every expansive action $G \times X \to X$ on a uniform space $(X,\U)$ without an isolated point is sensitive. \textit{Expansiveness} means that there exists $\eps \in \U$ such that for every distinct pair of point $x \neq y$ in $X$ we have $(gx,gy) \neq \eps$ for some $g \in G$. Many compact groups $K$ admit expansive automorphisms $\sigma \colon K \to K$. For instance, the cascade induced by a  hyperbolic toral automorphism is expansive (also topologically mixing), hence not Asplund representable. In fact, it is not even Rosenthal representable (see Theorem \ref{t:Lebed}).

\begin{example} \label{e:w-m} 
	The action  $\R^{\times} \times \R \to \R$ from Proposition \ref{t:n=1}.1 is Asplund representable by  Fact \ref{l:fromHNS}.2. Indeed, the enveloping semigroup of the 1-point compactification action is metrizable (see  Proposition \ref{t:n=1}.3). 
\end{example}

\sk 
Recall that an action of $G$ on $X$ is said to be:

(a) \textit{(algebraically) transitive} if for every $x,y \in X$ there exists $g \in G$ such that $gx=y$. 

(b) \textit{2-transitive} if 
for all ordered pairs $(x_1, x_2)$ and $(y_1, y_2)$ in $X$ with $x_1 \neq y_1$ and $x_2 \neq y_2$, there exists $g \in G$ such that $x_2 = gx_1, y_2= gy_1$.

(c) \textit{topologically transitive} if, for every pair of nonempty open subsets
$U$ and $V$ of $X$, there is an element $g\in G$ such that $gU \cap V$ is nonempty.

(d) \textit{weakly mixing} 
if the induced diagonal action of $G$ on $X \times X$ is topologically transitive. 

\begin{lemma} \label{l:2tr-wm} 
	Let $G \times X \to X$ be a 2-transitive action such that $X$ has no isolated point. Then this action is weakly mixing.
\end{lemma}
\begin{proof} 
	Let $O_1$ and $O_2$ be nonempty open subsets in $X \times X$. 
	Choose nonempty open rectangles $U_1 \times V_1 \subset O_1$ and $U_2 \times V_2 \subset O_2$.  
	Since $X$ has no isolated points, there exist $x_1 \in U_1, y_1 \in V_1, x_2 \in U_2, y_2 \in V_2$ such that 
	$x_1 \neq y_1$ and $ x_2 \neq y_2$. The 2-transitivity implies that $x_2 = gx_1, y_2= gy_1$ for some $g \in G$. Then clearly, $gO_1 \cap O_2$ is nonempty.  
\end{proof}

\begin{fact} \label{l:fromWMix}  \ 
\cite[Corollary 9.3]{GM-HNS} Let $(X,\U)$ be a uniform space and $X$ is a weakly mixing $G$-space which is non-sensitive with respect to $\U$. Then $X$ is trivial.  
\end{fact}

Note that the \textit{affine group} 
$\R^n \rtimes \GL_n(\R)$ can be embedded into $\GL_{n+1}(\R)$ as follows: 
$$\R^n \rtimes \GL_n(\R) \to \GL_{n+1}(\R), \ \ \ M \mapsto \ 
\left(\begin{array}{cc}
	M & v  \\
	0 & 1 
\end{array}\right)$$ 
where $M$ is an $n \times n$ matrix from $\GL_n(\R)$ and $v$ is an $n \times 1$ column.

\begin{prop} \label{t:GL} \
	\begin{enumerate} 
		\item The action of $\GL_2(\R)$ on $\R^2$ is not Asplund representable.
		\item  The conjugation action of the affine group $\R^2 \rtimes \GL_2(\R)$ (and hence, also of $\GL_3(\R)$) is not Asplund representable. 
		\end{enumerate}
\end{prop}
\begin{proof} (1) 
	Assuming the contrary, let $\nu \colon \R^2 \to K$ be a proper $\GL_2(\R)$-compactification of $\R^2$ such that $K$ is Asplund representable. Hence, the compact $\GL_2(\R)$-space $K$ is HNS (Fact \ref{l:fromHNS}.1). By definition of HNS, every (not necessarily compact) $G$-subspace $X$ is non-sensitive with respect to the induced (from $K$) precompact uniformity for every subgroup $G$ of $\GL_2(\R)$. 
	We claim that there exist a subgroup $G \subset \GL_2(\R)$ and a weakly mixing $G$-subspace $Y$ in $K$, where $Y$  topologically is homeomorphic to $\R$. By Fact \ref{l:fromWMix} this will provide the desired contradiction. 
	Consider 
	$$Y =\left\{\left( \begin{array}{cc}
		y  \\
		1  
	\end{array}  \right) \middle|  \ y\in \R \right\}, \ \ G =\left\{\left( \begin{array}{cc}
	a & b \\
	0 & 1  
\end{array}  \right) \middle|  \ b\in \R, a \neq 0 \right\} \simeq \R \rtimes  \R^{\times}.$$   
It is easy to see that the natural restricted action of $G$ on $Y$ is 2-transitive, hence also weakly mixing by virtue of Lemma \ref{l:2tr-wm}. Since $\nu$ is an equivariant and topological embedding, we conclude that $\nu(Y)$ is a weakly mixing $G$-subspace of $K$.

	(2) Use (1) and Remark \ref{r:semid}. 
	\end{proof}

\begin{thm} \label{t:SL} \ 
	\begin{enumerate}  
		\item The 
		action of $\SL_n(\Z)$ on $\R^n$ is not Asplund representable for every $n \geq 3$.   
		\item The conjugation action of $\SL_n(\R)$ is not Asplund representable for every $n \geq~4$. 
	\end{enumerate}
\end{thm}
\begin{proof} (1) 
	By a result of S.G. Dani and S. Raghavan
	\cite{DR}, the natural linear action of $\SL_n(\Z)$ on $\R^n$ is weakly mixing for every $n \geq 3$.  
	Now Facts \ref{l:fromHNS} and \ref{l:fromWMix} imply that this action is not Asplund representable.
	
	(2) The special affine group $\R^n \rtimes \SL_n(\R)$ can be embedded into $\SL_{n+1}(\R)$ as follows: 
	$$\R^n \rtimes \SL_n(\Z) \to \SL_{n+1}(\R), \ \ \ M \mapsto \ 
	\left(\begin{array}{cc}
		M & v  \\
		0 & 1 
	\end{array}\right)$$ 
	where $M$ is an $n \times n$ matrix from $\SL_n(\Z)$ and $v$ is an $n \times 1$ column. Now use (1). 
	\end{proof}

\sk 
\subsection{Rosenthal representability and the tameness}
\label{s:tame} 
For definitions of tame systems and the enveloping semigroup we refer to the introduction. 

%

\begin{fact} \label{f:Ref-Ros} \cite{GM-HNS,GM-rose} 
	For a compact metric $G$-space $X$ the following conditions are equivalent:
	\begin{enumerate}
		\item The $G$-space $X$ is representable on a (separable) Rosenthal Banach space.   
		\item The $G$-space $X$ is tame. 
		\item The enveloping semigroup $E(X)$, as a (compact) topological space, is Frechet. 
		\item $card (E(X)) \leq 2^{\aleph_0}$.
	\end{enumerate} 
\end{fact}

The following theorem was obtained very recently in a joint work with E. Glasner \cite{GM-TC}. We include the proof for the sake of completeness.  

\begin{thm} \label{t:LinAreTame} \cite{GM-TC} 
	The linear action $\GL_n(\R) \times \R^n \to \R^n$ 
	is Rosenthal representable. 
\end{thm}
\begin{proof} By results of \cite{GM-rose}, 
	 every Rosenthal representable $G$-space is embedded into a tame compact $G$-space and the tameness is preserved by the $G$-factors of compact $G$-spaces. Hence, by Fact \ref{f:Ref-Ros} 
	it is enough to show that the 1-point compactification $X:=\R^n \cup \{\infty\}$ is a tame $\GL_n(\R)$-system. Or, equivalently, that the cardinality of the enveloping semigroup $E(X)$ is not greater than $2^{\aleph_0}$. Let $p \in E(X)$. 
	Define 
	$$
	V_p:=p^{-1}(\R^n)=\{v \in \R^n : \ p(v) \in \R^n\}=\{v \in X : \ p(v) \neq \infty\}.
	$$
	\sk 
\noindent	\textbf{Claim:}  $V_p$ is a linear subspace (hence, closed) in $\R^n$ and the restriction $p| \colon V_p \to \R^n$ is a (continuous) linear map. 
	\sk 
	Indeed, let $g_i$ be a net in $T$ such that $\lim g_i=p$ in $E$. If 
	$u,v \in V_p$ 
	then $\lim g_iu=p(u) \in \R^n$ and $\lim g_iv=p(v) \in \R^n$. 
	Then by the linearity of maps $g_i$ we obtain 
	$$\lim g_i (c_1u+c_2v)=c_1p(u)+c_2p(v) \in \R^n.$$
	 Since  $V_p$ is finite dimensional, 
	the linear map $p|_{V_p}  \colon V_p \to \R^n$ is necessarily continuous.  	
	
	\sk 
	Using this claim we obtain that $card (E(X)) \leq 2^{\aleph_0}$. 
\end{proof}

\begin{remark} 
	The enveloping semigroup $E(X)$ of the action from Theorem \ref{t:LinAreTame}   
	can be identified with the semigroup of all partial linear endomorphisms 
	of $\R^n$.  
	To see this, observe that the claim from the proof of Theorem \ref{t:LinAreTame} can be reversed. 
	Namely, every \textit{partial} linear endomorphism $f \colon V \to W$ of $\R^n$ defines an element $p \in E(X)$ such that 
	$$
	V=p^{-1}(\R^n), \ p(V)=W, \ 
	p(x)=f(x) \ \forall x \in V, \ p(y)=\infty \ \forall y \notin V.
	$$
	Moreover, that assignment is a semigroup isomorphism: (partial) composition corresponds
	to the product of suitable elements from the enveloping semigroup. 
\end{remark} 

\begin{thm} \label{t:Lebed} 
	For every $n \geq 2$ 
	there exists a topological group automorphism $\sigma \colon \T^n \to \T^n$ 
	on the $n$-dimensional torus $\T^n$ 
	such that the corresponding 
	action of the cyclic group $\Z$ on $\T^n$ by the iterations 
	is not Rosenthal representable.  
\end{thm}
\begin{proof} 
	For every hyperbolic toral automorphism $\sigma \colon \T^n \to \T^n$, the corresponding cascade has positive entropy. Hence, it cannot be tame by a result of Kerr and Li \cite{KL} because tameness implies zero entropy. Therefore, such a cascade is not Rosenthal representable according to Fact \ref{f:Ref-Ros}.  
	\end{proof} 

\begin{remark} \label{r:Lebed}
	By a recent paper of V. Lebedev \cite{Leb}, for every $\sigma \in \SL_n(\Z)$ the corresponding cascade $\Z \times \T^n \to \T^n$ 
	 is tame (if and) only if $\sigma^k=I$ for some natural $k$.  So, in the proof of Theorem \ref{t:Lebed} one may consider any $\sigma \in \SL_n(\Z)$ with infinite order.  
\end{remark}

\begin{corol} \label{c:notTame1} 
	For the two dimensional Lie group $\T^2 \rtimes \Z$, its conjugation action is not Rosenthal representable. 
\end{corol}
\begin{proof} 
	Use Theorem \ref{t:Lebed} (for $n=2$) and take into account Remark \ref{r:semid}. 
	\end{proof}

\begin{thm} \label{t:notTame} 
	There exists a metrizable profinite (compact zero-dimensional) 
	 topological group $X$ and a topological group automorphism $a\colon X \to X$ such that the corresponding action of the cyclic group $\Z$ on $X$ is not Rosenthal representable. 
\end{thm}
\begin{proof} Consider a compact metrizable zero-dimensional $\Z$-space $K$ which is not tame. 
	For example, take the full Bernoulli shift $K:=\{0,1\}^{\Z}$. 
	It is well known that the enveloping semigroup of the $\Z$-space $K$ is $\beta \Z$. Hence, $K$ is not tame according to Fact \ref{f:Ref-Ros}. 

	Consider the \textit{free profinite group} $X:=F_{Pro}(K)$ of the Cantor space $K$. Then $F_{Pro}(K)$ is a compact \textbf{metrizable} group (see \cite{RZ} and  \cite[Theorem 5.1]{MeSh-FrNA13}). Moreover, we have a continuous action 
	$$\a \colon \Z \times F_{Pro}(K) \to F_{Pro}(K)$$ by group automorphisms. This action is not tame because its subaction $\Z \times K \to K$ is not tame. Therefore, the dynamical $\Z$-system $X$ is not Rosenthal representable. 
\end{proof} 

We denote by $S_{\infty}$ the symmetric topological group endowed with the pointwise topology acting on the discrete set $\N$.  
A topological group $G$ is \textit{non-archimedean} if open subgroups 
form 
a local topological basis. 

\begin{corol} \label{c:notTame} 
	There exists a locally compact Polish non-archimedean 
	group $G$ such that the $G$-space $G_c$ is not Rosenthal representable.
\end{corol}
\begin{proof} Let $\a \colon \Z \times F_{Pro}(K) \to F_{Pro}(K)$ be the action from the proof of Theorem \ref{t:notTame}.   The corresponding topological semidirect product $G:=X \rtimes_{\a} \Z$  is the desired group, where $X$ is a $\Z$-group from Theorem \ref{t:notTame}. Indeed, the conjugation action $G \times G_c \to G_c$ of this group is not Rosenthal representable (otherwise, by Remark \ref{r:semid}, the same is true for its subaction $\a \colon \Z \times X \to X$).  
	\end{proof}

This locally compact group $G$ from Corollary \ref{c:notTame} is an extension of a profinite group by a discrete group. In particular, $G$ is \textit{elementary} in the sense of Wesolek \cite{Wesolek}. 

\begin{corol}
	For the symmetric topological group $G=S_{\infty}$, the conjugation action on $G_c$ is not Rosenthal representable.
	
\end{corol}
\begin{proof}
	(3) Use Corollary \ref{c:notTame} and the well-known fact (see \cite[Theorem 1.5.1]{BK}) that every second countable non-archimedean group $G$ is embedded into $S_{\infty}$. 
\end{proof}

In contrast, the left action of $S_{\infty}$ on itself is always Hilbert representable and the topological group $S_{\infty}$ is unitarily representable on a Hilbert space.

\begin{quest}
	Is it true that for $G=\SL_2(\R)$ the $G$-space $G_c$ is Rosenthal 
	representable?
\end{quest}

\sk 
\section{Banach representations of homogeneous actions} 
\label{s:homog} 

For every locally compact second countable group $G$, its regular left action on itself is Hilbert representable. 
However, as we already have seen, the conjugation action of $G$ might not even be  Rosenthal-representable.  
For homogeneous $G$-spaces $G/H$, we have the same phenomenon (for nontrivial $H$).  Our results above shed some light also on Banach representations of homogeneous actions. 
Representability of the $G$-space $G=G/\{e\}$ on nice Banach spaces sometimes  cannot be decisive  about the representations of the $G$-space $G/H$ even for locally compact, as well as for Lie, groups $G$.

\begin{corol} \label{c:homogLie} 
	Let $G:=\T^2 \rtimes \Z$ be the Lie group from Corollary \ref{c:notTame1}. Then for its cocompact discrete subgroup $H:=\Z$, the compact two dimensional homogeneous $G$-space $G/H$ is not Rosenthal representable. 
\end{corol}
\begin{proof}
	The original (nontame) action $\Z \times \T^2 \to \T^2$ 
	from Theorem \ref{t:Lebed} is a subaction of the homogeneous action $G \times G/\Z \to G/\Z$ (Fact \ref{f:SemidCoset}). So, the compact $G$-space $G/\Z$ is also nontame.  	
	\end{proof} 

\begin{corol}
	There exists a non-archimedean elementary locally compact group $G$ and a cocompact discrete subgroup $H$ such that the compact coset $G$-space $G/H$ is not Rosenthal representable. 
\end{corol}
\begin{proof}
	Let $G:=X \rtimes_{\a} \Z$ be the group from Corollary \ref{c:notTame} with $H:=\Z$. For the rest of the proof, as above, apply Fact \ref{f:SemidCoset}. 
\end{proof}

\begin{corol}
	There exists a closed subgroup $H$ of $G:= \SL_2(\R)$ such that the corresponding locally compact coset $G$-space $G/H$ is not reflexively representable. 
\end{corol}
	\begin{proof}
			
		Indeed, 
		in Theorem \ref{t:SL(2)} we deal, in fact, with the restricted action by conjugations of $G:= \SL_2(\R)$ on the $G$-orbit of $x_1$. 
		This orbit, containing every 
		$$x_m:= \left( \begin{array}{cc}
			1 & m  \\
			0 & 1  
		\end{array}\right)$$
		(in terms of Theorem \ref{t:SL(2)}) is locally compact and it can be represented as the topological coset $G$-space $G/H$, where $H$ is the stabilizer subgroup 
		$$St(x_1)=\{g \in  \SL_2(\R): gx_1g^{-1}=x_1\}.$$ 
		\end{proof} 

Results of the present section suggest the following general questions. 

\begin{quest} \label{q:RosCoset} 
	Which interesting homogeneous $G$-spaces $G/H$ are Rosenthal representable?  
		What about the  $\SL_n(\R)$-spaces $\SL_n(\R) / H$ ? In particular, what if $H=\SL_n(\Z)$?
\end{quest}


The ``Asplund version" of Question \ref{q:RosCoset} seems to be less attractive than the present ``Rosenthal version". The reason is that 
 by \cite[Theorem 6.10]{Me-nz} and Fact \ref{l:fromHNS}, Asplund representable compact metrizable coset $G$-spaces $G/H$ are necessarily equicontinuous. Therefore they are even Hilbert representable.   
 Many geometric compact coset $G$-spaces are sensitive. Hence, not Asplund representable.  
For example, this is true for the projective action which is 2-transitive. On the other hand, results of R. Ellis \cite{E}  show that the action of 
$G =\mathrm{GL}(d,\R)$ on the projective space
$\mathbb{P}^{d-1}, \ d \ge 2$, is tame (hence, Rosenthal representable).

\bibliographystyle{alpha}

\end{document}